\documentclass[11pt, reqno]{amsart}
\usepackage{amsmath,amsfonts,amssymb,amsthm,color,cancel}
\usepackage{epsfig}
\usepackage[normalem]{ulem}

\usepackage{bbm}
\usepackage{enumerate}
\usepackage{amstext}
\usepackage{mathrsfs}
\usepackage{xspace}
\usepackage{amsfonts}
\usepackage{amsmath}
\usepackage{amssymb}
\usepackage{amstext}
\usepackage{amsthm}       %proof environment :)
\usepackage{xspace}
\usepackage[active]{srcltx}
\usepackage{cite}

\newtheorem{theorem}{Theorem}[section]
\newtheorem{lemma}[theorem]{Lemma}

\newtheorem{definition}[theorem]{Definition}
\newtheorem{proposition}[theorem]{Proposition}
\newtheorem{remark}[theorem]{Remark}
\newtheorem{corollary}[theorem]{Corollary}

\newcommand{\gm}{\gamma}
\newcommand{\M}{\mathcal{M}}
\newcommand{\R}{\mathbb{R}}
\newcommand{\cS}{\mathcal{S}}

\newcommand{\T}{\mathbb{T}}

\newcommand{\A}{\mathcal{A}}

\newcommand{\ri}{\rightarrow}
\newcommand{\eps}{\varepsilon}
\newcommand{\Proof}{\begin{proof}}
\newcommand{\End}{\end{proof}}

\newcommand{\Dr}[1]{\mbox{\rm #1}}
\newcommand{\Lip}{\Dr{Lip}}

\numberwithin{equation}{section}

\begin{document}

%%%%% To ease editing, add:

%\baselineskip=15pt

%%%%%%%%%%%%%%%%

\title{On the vanishing discount problem\\ from the negative direction}
%\titlerunning{On the vanishing discount problem from the negative direction}

\author{Andrea Davini \and Lin Wang}
\address{Dip. di Matematica, {Sapienza} Universit\`a di Roma,
P.le Aldo Moro 2, 00185 Roma, Italy}

\email{davini@mat.uniroma1.it}
\address{Yau Mathematical Sciences Center, Tsinghua University, Beijing 100084, China}

\email{linwang@tsinghua.edu.cn}
%\keywords{Homogenization, equations in media with random structure, nonconvex Hamilton-Jacobi equation.}

%%%%%%%%
\begin{abstract}
\noindent It has been proved in \cite{DFIZ1} that the unique viscosity solution of
\begin{equation}\label{abs}\tag{*}
\lambda u_\lambda+H(x,d_x u_\lambda)=c(H)\qquad\hbox{in $M$},
\end{equation}
uniformly converges, for $\lambda\rightarrow 0^+$, to a specific solution $u_0$ of the critical equation
\[
H(x,d_x u)=c(H)\qquad\hbox{in $M$},
\]
where $M$ is a closed and connected Riemannian manifold and $c(H)$ is the critical value.
In this note, we consider the same problem for $\lambda\rightarrow 0^-$.
In this case,  viscosity solutions of equation \eqref{abs} are not unique, in general, so we focus
on the asymptotics of the minimal solution $u_\lambda^-$ of \eqref{abs}.
Under the assumption that constant functions are subsolutions
of the critical equation, we prove that the $u_\lambda^-$ also converges to $u_0$
as $\lambda\rightarrow 0^-$.
Furthermore, we exhibit an example of $H$ for which equation \eqref{abs} admits a unique solution for $\lambda<0$ as well.
\end{abstract}

\date{\today}
\maketitle

%\tableofcontents

%\newpage
%%%%%%%%%%%%%%%%%%%%%%%%%%%%%%%%%%%%%%%%%%%%%%%%%%%%%%%%Sect. 1

\section{Introduction and main results}
\setcounter{equation}{0}
\setcounter{footnote}{0}
Let $M$ be a connected, closed smooth Riemannian manifold and $H: T^*M\rightarrow \R$  a $C^3$ Tonelli Hamiltonian, where Tonelli refers to the fact that $H$ is strictly convex and superlinear with respect to $p$. We consider the Hamilton-Jacobi equation:
\begin{equation}\label{aa}\tag{A$_\lambda$}
-\lambda u+H(x,d_x u)=c(H)\qquad\hbox{in $M$},
\end{equation}
where $\lambda$ is a positive parameter and $c(H)$ is the critical value, given by \cite{CIPP}
\[c(H)=\inf_{u\in C^\infty(M,\R)}\sup_{x\in M}H(x,d_x u).
\]
We are interested in understanding the asymptotics of the solution(s) to \eqref{aa} as $\lambda\to 0^+$. The problem is well understood when $\lambda\to 0^{-}$,
see \cite{DFIZ1}:
when $\lambda<0$, in fact, equation \eqref{aa} admits a unique viscosity solution and the latter converges, as $\lambda\to 0^-$, to a specific
solution of the associated critical equation
\begin{equation}\label{ss}\tag{A$_0$}
H(x,d_x u)=c(H)\qquad\hbox{in $M$}.
\end{equation}
The interest of the result relies on the fact that  the critical equation \eqref{ss}  admits infinite solutions, even up to additive constants in general.
We also refer the reader to \cite{CCIZ,DFIZ1,Go,Go1,IsMiTr-discount1,IsMiTr-discount2,MiTr-discount,ZC} for related results.

When $\lambda>0$, the uniqueness of the viscosity solution to \eqref{aa} fails. For example, see \cite[Example 1.1]{WWY2}, the function $u_1\equiv 0$ and the 1--periodic function $u_2$
satisfying $u_2(x)=x^2/2$ for $x\in [-1/2,1/2]$ are both viscosity solutions of the equation
\[
-u(x)+\frac{1}{2}|u'(x)|^2=0,\quad x\in\T^1:=\mathbb{R}/\mathbb{Z}.
\]
Due to this nonuniqueness phenomenon, we will consider the vanishing discount problem for the minimal viscosity solutions of
\eqref{aa}. More precisely, let $\mathcal{S}_\lambda^-$ be the set of viscosity solutions of (\ref{aa}) and denote by
\[{u}^-_\lambda(x):=\min_{v \in \cS^-_\lambda}v (x).
\]
The function $u^{-}_\lambda$ is a Lipschitz continuous viscosity solution of \eqref{aa} as well, see \cite[Theorem 1.2]{WWY2}.
The asymptotic convergence is established under the assumption that constant functions are
subsolution of the critical equation \eqref{ss}.
%Due to the results obtained in \cite{DFIZ1,WWY,WWY1,WWY2},

%
% Let $\mu$ be the projected Mather measure of $H$ and we use
% $\hat{F}^-$  to denote the set of viscosity solutions of (\ref{ss}) with $\int_{M}u(x)d\mu(x)\leq 0$.
\begin{theorem}\label{l11}
Let us assume that $H(x,0)\leq  c(H)$ for every $x\in M$. Then
${u}^-_\lambda$ converges to $u_0^-$ uniformly on $M$ as $\lambda\ri 0^+$, where $u_0^-$ is the unique viscosity solution of \eqref{ss}
such that $u_0^-\equiv 0$ on the projected Aubry set $\A$ associated with \eqref{ss}.
\end{theorem}

We point out that $u^-_0$ is the same function that we obtain when we study the asymptotics of the solutions of \eqref{aa} for $\lambda\to 0^-$,
see \cite{DFIZ1} or Theorem \ref{darre} below.

When (\ref{aa}) has a unique viscosity solution for each $\lambda<0$, Theorem \ref{l11} yields in particular that this solution converges to $u_0^-$ uniformly
on $M$ as
$\lambda\ri 0^+$.  It is worth mentioning that uniqueness of viscosity solutions of \eqref{aa} still holds under certain dynamical assumption. For example,
the equation
\begin{equation}\label{ex-unique88}\tag{E$_D$}
-\lambda u+\frac{1}{2}|d_x u|^2+U(x)=c,\quad x\in \T^1:=\mathbb{R}/\mathbb{Z},\quad \lambda>0,
\end{equation}
where $U:\T^1\to\R$ is of class $C^3$ and has a unique maximum point $x_0$ with $U(x_0)=c$ and $U''(x_0)<0$, has a unique viscosity
solution,  when $\lambda>0$ is small enough, see Section \ref{e888} below.

By \cite[Proposition 2.8]{WWY2}, $v$ is a backward (resp. forward) weak KAM solution of equation \eqref{aa} if and only if $-v$ is a forward (resp. backward)
weak KAM solution of equation:
\begin{equation}\label{bb}\tag{B$_\lambda$}
\lambda u+H(x,-d_x u)=c(H)\qquad\hbox{in $M$},
\end{equation}
where backward weak KAM solutions and  viscosity solutions are the same. The same holds for $\lambda=0$ as well.

Let $\mathcal{S}_\lambda^+$ be the set of forward weak KAM solutions of (\ref{bb}) and denote
\[
{u}_\lambda^+(x):=\sup_{v \in \cS^+_\lambda}v (x).
\]
Based on the correspondence between viscosity solutions of (\ref{aa}) and forward weak KAM solutions of (\ref{bb}), Theorem \ref{l11} is equivalent to the following

\begin{theorem}\label{222}
Let us assume that $H(x,0)\leq  c(H)$ for every $x\in M$. Then  ${u}^+_\lambda$ converges to $u_0^+$ uniformly on $M$ as $\lambda\ri 0^+$,
where $u_0^+$ is the unique forward weak KAM solution of
\begin{equation}\label{eq HJ check}\tag{B$_0$}
H(x,-d_x u)=c(H)\qquad\hbox{in $M$}
\end{equation}
such that $u_0^+\equiv 0$ on the projected Aubry set $\check\A$ associated with \eqref{eq HJ check}.
\end{theorem}
We point out that $\check \A=\A$, as we will see in Section \ref{sez generalities}.

Our analysis is based on an extension of some aspects of Aubry-Mather theory to contact Hamiltonian systems, as developed in
\cite{SWY,WWY,WWY1,WWY2},
for which discounted Hamilton-Jacobi equations serve as special models.
Other output in this vein can be found in a series of papers including \cite{CCY, CCJWY, MS, MK, WY}.

The condition that constant functions are subsolutions of the critical equation \eqref{ss},
under which the asymptotic convergence is established, is for instance satisfied whenever the Hamiltonian $H$ is
{\em reversible,} i.e. $H(x,p)=H(x,-p)$ for all $(x,p)\in T^*M$, see for instance Proposition \ref{prop reversible} below.
The model example is the mechanical Hamiltonian $H(x,p)=\frac{|p|_x^2}{2}+U(x)$.

Another example is provided by the Ma\~{n}\'{e}'s Hamiltonian $H(x,p):=\frac{1}{2}|p|^2_x+\langle p, X\rangle_x$,
where  $X:M\ri TM$ is a smooth vector field and $\langle \cdot, \cdot\rangle_x$ denotes the standard inner product in $T^*_xM$.
Clearly, constant functions are solutions of the equation \ $H(x,d_x u)=0$\ \ in $M$.

Other examples of Hamiltonians for which constant functions are critical subsolutions can be obtained in the following way:
according to \cite{BernardC11}, any $C^3$--Tonelli Hamiltonian $H$
admits a $C^{1,1}$--critical subsolution $\varphi$.
The new Hamiltonian $\tilde H(x,p):=H(x,d_x\varphi +p)$ satisfies
\ \ $\tilde H(x,0)\leq c(H)=c(\tilde H)$\ \ for every $x\in M$. In order to apply our results to $\tilde H$, we need the latter to
be of class $C^{3}$, meaning we need the existence of a critical subsolution of class $C^4$ on $M$. This is true under
proper dynamical assumptions, see \cite{BernardSmooth}, but it is not a general fact, see \cite[Appendix A]{BernardC11}.

%and requires $C^3$--regularity of the Hamiltonian. Other output in this vein can be found in a series of papers including \cite{CCY, CCJWY, MS, MK, WY}-
%

%

\section{Generalities}\label{sez generalities}

In this paper, we assume the Hamiltonian $H:T^*M\to\R$ to be of class $C^3$ and to satisfy the following assumptions:

\begin{itemize}
\item[(H1)] {\bf (strict convexity)} \quad $\displaystyle\frac{\partial^2 H}{\partial p^2} (x,p)$\quad  is positive definite as a quadratic form on $T^*_x M$, for every $x\in M$;
\item[(H2)] {\bf (superlinearity)}
\[
\inf_{x\in M} \frac{H(x,p)}{|p|_x}\to +\infty\qquad\hbox{as $|p|_x\to +\infty$.}
\]
\end{itemize}
We will denote by $L:TM\to\R$ the Lagrangian associated with $H$ via the Fenchel transform. The Lagrangian $L$ is of class $C^3$ and satisfies assumptions analogous
to (H1)-(H2). We point out that the $C^3$--regularity is  only required in order to apply the results of \cite{SWY,WWY,WWY1,WWY2}.

Let $H$ be a Hamiltonian satisfying the above assumptions and let us consider an Hamilton--Jacobi equation of the form
\begin{equation}\label{eq discounted}
\lambda u+H(x,-d_x u)=a\qquad\hbox{in $M$},
\end{equation}
where $a\in\R$ and $\lambda\geq 0$. For the notion of viscosity (sub-, super-) solution of \eqref{eq discounted}, we refer to \cite{barles}.
Viscosity (sub-, super-) solutions
will be always assumed continuous in the sequel, with no further specification.
Set $\check H(x,p):=H(x,-p)$ and denote by $\check L$ the associated Lagrangian.

\subsection{Subsolutions and the critical value}
Due to conditions (H1)-(H2), the following equivalence holds, see for instance \cite{DFIZ1, FSC1, Fa12} and references therein.
\begin{proposition}
Let $\lambda\geq 0$ and $v\in\Dr{C}(M,\R)$. The following are equivalent facts:
\begin{itemize}
 \item[(i)] $v$ is a viscosity subsolution of  \eqref{eq discounted};
 \item[(ii)] $v\in \Lip(M,\R)$ and it is  an almost everywhere subsolution of \eqref{eq discounted}, i.e.
 $$
 \lambda v(x)+H(x,-d_x v)\leq a\qquad \hbox{for a.e. $x\in M$;}
 $$
 \item[(iii)] for every absolutely continuous curve $\gamma:[t_1,t_2]\to M$ we have
 \begin{equation}\label{eq dominating curve}
  e^{\lambda t_2}v(\gamma(t_2))-e^{\lambda t_1} v(\gamma(t_1)) \leq  \int_{t_1}^{t_2} e^{\lambda s}\big( \check L(\gamma(s),\dot\gamma(s))+a \big) ds.
 \end{equation}
\end{itemize}
\end{proposition}
 Let $\gamma:[t_1,t_2]\to\R$ be a curve for which \eqref{eq dominating curve} holds with an equality for a subsolution $v$ of \eqref{eq discounted}.  When
$v$ is differentiable at $x=\gamma(a)$, we have that $\gamma$ is the projection on $M$ of an integral curve of the discounted flow { generated by
 \begin{align}\label{cott}\tag{DH}
        \begin{cases}
        \dot{x}=\frac{\partial {\check H  }}{\partial p}(x,p),\smallskip\\
        \dot{p}=-\frac{\partial {\check H  }}{\partial x}(x,p)- \lambda p.
         \end{cases}
\end{align}
Namely, $\gamma(t)=\pi\big(\check \Phi^{t-a}_\lambda(x,d_x v)\big)$ for all
$t\in [t_1,t_2]$, where $\pi:T^*M\to M$ denotes the standard projection and $\check \Phi_\lambda^{t}$ denotes the discounted flow generated by (\ref{cott}).}

When $\lambda>0$, equation \eqref{eq discounted} admits a unique solution for every $a\in\R$. When $\lambda=0$, on the other hand, there exists a
a unique real constant $c(\check H)$, hereafter called {\em critical}, for which the equation
\begin{equation*} %\label{eq HJ check bis}
H(x,-d_x u)=c(\check H)\qquad\hbox{in $M$}
\end{equation*}
admits viscosity solutions. Such a critical constant $c(\check H)$ is also characterized as follows:
\begin{equation}\label{def critical constant}
c(\check H)=\min\{a\in\R\,\mid\, \exists\, v\in\Lip(M,\R)\ \hbox{such that}\ \ \check H(x,d_xv)\leq a\ \ \hbox{for a.e. $x\in M$}\}.
\end{equation}
Since $v\in\Lip(M,\R)$ is an a.e. subsolution of $\check H(x,d_x v)=a$ in $M$ if and only if $-v$ is an a.e. subsolution of
$H(x,d_x v)=a$ in $M$, it is clear from \eqref{def critical constant} that $c(\check H)=c(H)$.

\subsection{Forward weak KAM solutions and minimizing sets}
Let  $\check L:TM\to\R$ denotes the Lagrangian associated with ${\check H}$ via the Fenchel transform.
 The following definition is an adaptation to the case of discounted equations of the notion of forward weak KAM solutions, first introduced by Fathi \cite{Fat-b}
for the non-discounted Hamilton-Jacobi equation and subsequently generalized in \cite[Definition 2.2]{WWY2} for fairly general contact Hamiltonian systems.

\begin{definition}\label{bwkam}
	A function $v\in C(M,\mathbb{R})$ is called a forward weak KAM solution of \eqref{bb} if
	{ \begin{itemize}
		\item [(i)] for each continuous piecewise $C^1$ curve $\gamma:[t_1,t_2]\rightarrow M$, we have
			\begin{equation}\label{cali11}
		e^{\lambda t_2}v(\gamma(t_2))-e^{\lambda t_1}v(\gamma(t_1))\leq\int_{t_1}^{t_2}  e^{\lambda s}\big( {\check L}(\gamma(s),\dot{\gamma}(s))+c(H)\big)\,ds;
		\end{equation}
		\item [(ii)] for each $x\in M$, there exists a $C^1$ curve $\gamma:[0,+\infty)\rightarrow M$ with $\gamma(0)=x$ such that
		\begin{align}\label{cali2}
		e^{\lambda t}v(\gm(t))-v(x)=\int_{0}^{t} e^{\lambda s}\big( {\check L}(\gamma(s),\dot{\gamma}(s))+c(H)\big)\,ds\qquad \hbox{for all $t>0$.}
		\end{align}
	\end{itemize}}
\end{definition}

Let $\check{L}_\lambda(x,\dot{x},u):=-\lambda u+\check{L}(x,\dot{x})$. Curves satisfying \eqref{cali2}  are called  $(v,{\check L}_\lambda,c(H))$-calibrated curves. We denote by $\mathcal{S}_\lambda^+$  the set of forward weak KAM solutions of (\ref{bb}). {Based on the forward Lax-Oleinik semigroup introduced in \cite{WWY1}, we have
\begin{equation}\label{2-3}
\check{T}^+_{t,\lambda}\varphi(x)=\sup_{\gamma(0)=x}\left\{e^{\lambda t}\varphi(\gamma(t))-\int_0^te^{\lambda s }\big( \check L(\gamma(s ),\dot{\gamma}(s ))+c(H)\big) ds \right\},
\end{equation}}

Let us denote by $u_\lambda$ the unique viscosity solution of \eqref{bb}.
 The following proposition will be employed to show uniqueness of viscosity solutions to(\ref{ex-unique88}).
\begin{proposition}[\cite{WWY2}]\label{wwyrr}
Let $v\in \cS^+_\lambda$. { The following holds:
\begin{itemize}
\item [(i)] Denote by $\mathcal{I}_v^\lambda:=\{x\in M\ |\ v(x)=u_\lambda(x)\}$. Then both $v$ and $u_{\lambda}$ are of class $C^{1,1}$ on $\mathcal{I}_v^\lambda$.
\item [(ii)] Denote by $\tilde{\mathcal{I}}_v^\lambda:=\{(x,p)\in T^*M\,\mid\, v(x)=u_{\lambda}(x),\ p=d_x v=d_x u_{\lambda}\}$. Then
$\tilde{\mathcal{I}}_v^\lambda$ is a non-empty and compact invariant set by the discounted flow $\check\Phi_\lambda^t$ generated by $\check H$. Furthermore,
if we denote by
\[
(x(t),p(t)):=\check\Phi_\lambda^t(x_0,p_0) \qquad \hbox{for all $t\in \R$}
\]
then for each $(x_0,p_0)\in \tilde{\mathcal{I}}_v^\lambda$, we have $p(t)=d_{x(t)}u_{\lambda}$ for each $t\in \R$.

\item [(iii)] Given $x_0\in M$, let $\gm:[0,+\infty)\rightarrow M$ be a $(v, \check{L}_\lambda, c(H))$-calibrated curve with $\gm(0)=x_0$.
Let $p_0:=\frac{\partial {\check L}}{\partial \dot{x}}(x_0,\dot{\gm}(0^+))$, where $\dot{\gm}(0^+)$ denotes the right derivative of
$\gm(t)$ at $t=0$. Then $d_{\gm(t)}v$ exists for every $t>0$ and
\[(
\gm(t),d_{\gm(t)}v)=\check\Phi_\lambda^t({x}_0,{p}_0).
\]
Furthermore \ \ $\omega(x_0,p_0)\subseteq \tilde{\mathcal{I}}^\lambda_{v}$,\ \ where $\omega(x_0,p_0)$ denotes the $\omega$-limit set of $(x_0,p_0)$ with respect
to the discounted flow $\check\Phi_\lambda^t$.
\end{itemize}}
\end{proposition}

Let $w$ be a Lipschitz continuous function and set
{\[
G_{w}:=\overline{\{ (x,p)\in T^*M\,\mid\,\hbox{$w$ is differentiable at $x$, $p=d_x w$}\, \}}.
\]
}
By \cite[Theorem 1.1]{WWY2}, $G_{u_\lambda}$ and $G_v$ for each $v\in \mathcal{S}_\lambda^+$ are backward and forward invariant by $\check \Phi_\lambda^{t}$,   respectively. Define
\begin{equation}\label{def sigma}
\tilde {\mathcal{A}}^\lambda:=\bigcap_{t\geq  0} \check \Phi_\lambda^{-t}\left(G_{u_\lambda}\right),\qquad {\mathcal{A}}^\lambda :=\pi\big(\tilde{\mathcal{A}}^\lambda \big),
\end{equation}
where {$\pi:T^*M\to M$ denotes the standard projection}. The sets $\tilde {\mathcal{A}}^\lambda$ and ${\mathcal{A}}^\lambda$ are called Aubry and
projected Aubry set associated with $\check{H}_\lambda-c(H)$ where $\check{H}_\lambda(x,p,u):=\lambda u+\check{H}(x,p)$, respectively. According to \cite[Theorem 1.1]{WWY2}, the set $\tilde{\mathcal{A}}^\lambda $ is invariant under
$\left(\check \Phi^t_\lambda\right)_{t\in\R}$.
According to \cite[Theorem 1.2]{WWY2}, we have the following result.
\begin{theorem}\label{teo forward sol}
{The limit $\lim_{t\rightarrow+\infty}\check{T}^+_{t,\lambda}u_\lambda$ exists. Let
\[u_\lambda^+:=\lim_{t\rightarrow+\infty}\check{T}^+_{t,\lambda}u_\lambda.\]}
Then $u_\lambda^+=\check{T}_{t,\lambda}^+u_\lambda^+$ for each $t\geq 0$. Moreover, the following holds:
\begin{itemize}
{\item[(i)] ${u}_\lambda^+$ is the maximal forward weak KAM solution of \eqref{bb};}

\item[(ii)] $u_\lambda\geq  {u}_\lambda^+$ in $M$ \ \  and\ \  ${u}_\lambda={u}_\lambda^+$ on $\mathcal{A}^\lambda $;
\item[(iii)] $\tilde{\mathcal{A}}^\lambda=\tilde{\mathcal I}_{u_\lambda^+}^\lambda$\ \  and\ \  ${\mathcal{A}}^\lambda=\mathcal{I}_{u_\lambda^+}^\lambda$.
\end{itemize}
\end{theorem}

Let $v\in\cS^+_\lambda$.
 We will call {\em Mather measure} associated with $v$ any  Borel $\check\Phi_\lambda^t$-invariant probability measures supported in $\tilde{\mathcal{I}}_v^\lambda$. We shall denote by $\mathfrak{M}_v$ the set of such measures. Since $\tilde{\mathcal{I}}_v^\lambda$ is a non-empty and compact invariant set,
{Krylov-Bogoliubov's theorem}, see \cite[\S II, Theorem I]{KB}, guarantees that $\mathfrak{M}_v$ is non-empty.
The {Mather set} associated with $v$ is given by
\[
\tilde{\mathcal{M}}_v^\lambda=\bigcup_{\mu\in \mathfrak{M}_v}\text{supp}(\mu),
\]
where $\text{supp}(\mu)$ denotes the support of $\mu$.  A classical argument shows that there exists $\nu\in \mathfrak{M}_v$ such that
$\tilde{\mathcal{M}}_v^\lambda=\text{supp}(\nu)$, yielding that $\tilde{\mathcal{M}}_v^\lambda$ is also a non-empty and compact invariant set.
Indeed, it suffices to define $\nu$ as a convex combination of a dense sequence of measures in $\mathfrak{M}_v$.  
Following \cite{WWY2}, the {Mather set} of  (\ref{cott}) is defined as
\begin{equation}\label{mmaa}
\tilde{\mathcal{M}}^\lambda:=\tilde{\mathcal{M}}_{{u}_\lambda^+}^\lambda,
\end{equation}
where $u^+_\lambda$ denotes the maximal forward weak KAM solution of \eqref{ex-unique}.
Such a Mather set $\tilde{\mathcal{M}}^\lambda$ is maximal, in the sense that
$\tilde{\mathcal{M}}_{{v}}^\lambda\subseteq \tilde{\mathcal{M}}_{{u}_\lambda^+}^\lambda$  {and $v=u_\lambda$ on ${\mathcal{M}}^\lambda_v$} for each $v\in \cS^+_\lambda$.
This is a straightforward consequence of the fact that $\tilde{\mathcal{I}}_v^\lambda\subseteq \tilde{\mathcal{I}}_{u_\lambda^+}^\lambda$ {and $v=u_\lambda$ on ${\mathcal{I}}_v^\lambda$ in view of  Proposition \ref{wwyrr}} and
Theorem \ref{teo forward sol}.\smallskip

 The following holds.

\begin{proposition}\label{asypp}
Let $v\in \cS_+$, $x_0\in M$ and $\gm:[0,+\infty)\rightarrow M$ be a $(v, \check{L}_\lambda, c(H))$-calibrated curve with $\gm(0)=x_0$.
Let   $p_0:=\frac{\partial \check{L}}{\partial \dot{x}}(x_0,\dot{\gm}(0^+))$, where $\dot{\gm}(0^+)$ denotes the right derivative of $\gm(t)$ at $t=0$.
Then
\[
\omega(x_0,p_0)\cap \tilde{\mathcal{M}}^\lambda_{v}\not=\emptyset,
\]
where $\omega(x_0,p_0)$ denotes the $\omega$-limit set of $(x_0,p_0)$ with respect
to the contact Hamiltonian flow $\check\Phi_\lambda^t$.
\end{proposition}

\begin{proof}
Let us assume by contradiction that\ \   $\omega(x_0,p_0)\cap \tilde{\mathcal{M}}_v^\lambda=\emptyset$.\
Since $\omega(x_0,p_0)$  is a non-empty, compact and invariant subset of $\tilde{\mathcal{I}}_v^\lambda$, see Proposition \ref{wwyrr},
by applying Krylov-Bogoliubov's theorem \cite[\S II, Theorem I]{KB} we would find a new Mather measure supported on $\omega(x_0,p_0)$, in contradiction with the
very definition of $\tilde{\mathcal{M}}_v^\lambda$.
\end{proof}

As a consequence, we derive the following result.

\begin{proposition}\label{copp}
Let $v_1$, $v_2$ be forward weak KAM solutions of \eqref{bb}.
If  $v_1|_\mathcal{O}= v_2|_\mathcal{O}$ ,
%\[
%\mathcal{I}_{v_1}=\mathcal{I}_{v_2},\quad v_1|_\mathcal{O}= v_2|_\mathcal{O},
%\]
where  $\mathcal{O}$ denotes a neighborhood of $\mathcal{M}_{v_1}$, then $v_1 \leq v_2$ in $M$.
\end{proposition}
\begin{proof}
Pick $x_0\in M$ and let $\gm:[0,+\infty)\rightarrow M$ be a $(v_1 , \check{L}_\lambda, c(H))$-calibrated curve with $\gm(0)=x_0$.
Let   $p_0:=\frac{\partial \check{L}}{\partial \dot{x}}(x_0,\dot{\gm}(0^+))$, where $\dot{\gm}(0^+)$ denotes the right derivative of $\gm(t)$ at $t=0$.

Based on Proposition \ref{asypp}, there exists a $t>0$ such that $\gamma(t)\in\mathcal O$.
From the hypothesis that $v_1 =v_2$ on $\mathcal{O}$ we infer that $v_1(\gamma(t))=v_2(\gamma(t))$.
From the fact that $v_i=T^+_{t,\lambda}v_i$ on $M$ for $i\in\{1,2\}$ and \eqref{2-3} we get
\[
e^{\lambda t} v_1(\gamma(t))-v_1(x_0)
=
\int_0^t e^{\lambda s }\big( \check L(\gamma(s ),\dot{\gamma}(s ))+c(H)\big) ds
\geq
e^{\lambda t} v_2(\gamma(t))-v_2(x_0),
\]
yielding $v_1(x_0)\leq v_2(x_0)$ since $v_1(\gamma(t))=v_2(\gamma(t))$. The assertion follows since $x_0$ was arbitrarily chosen in $M$.
\end{proof}

\subsection{Viscosity solutions of \eqref{eq HJ check}}
Viscosity solutions to equation \eqref{eq HJ check} are not unique, even up to additive constants in general.
A uniqueness set for equation \eqref{eq HJ check} is given by the so-called {\em projected Aubry set} $\check \A$: it {is} a closed subset of $M$ that can be
characterized by the following property, see \cite{FSC1}:
\begin{equation} \label{def Aubry}\tag{$\A$}
y\in \check\A \qquad\hbox{iff}\qquad \hbox{any subsolution $v$ of \eqref{eq HJ check} is differentiable at $y$.}
\end{equation}
Since $v$ is a subsolution of \eqref{eq HJ check} if and only if $-v$ is a subsolution of \eqref{ss},
we easily derive that $\check \A=\A$, where $\A$ is the projected Aubry set associated with equation \eqref{ss}.
In the sequel, we will always write $\A$ in place of $\check \A$. The following properties hold, see for instance \cite{FS05, FSC1, Fat-b}:\\

\begin{proposition}\label{prop Aubry}
\
\begin{itemize}
\item[(i)] Let $y\in\A$. Then \ $H(y,-d_{y}v)=c(H)$\ \ and\ \  $d_y v=d_y w$ \quad for any pair $v,w$ of subsolutions  to \eqref{eq HJ check}.
\item[(ii)] Let $y\in\A$. Then there exists a unique curve $\gamma:\R\to \A\subseteq M$ with $\gamma(0)=y$ such that
\[
v(\gamma(b))-v(\gamma(a))=\int_a^b \big( \check L(\gamma(s),\dot\gamma(s))+c(H) \big) ds\qquad\hbox{for every $a<b$,}
\]
for any subsolution $v$ of \eqref{eq HJ check}.
\item[(iii)] Let $u,v$ be viscosity solutions of \eqref{eq HJ check}. If $u=v$ on $\A$, then $u=v$ on $M$.
\end{itemize}
\end{proposition}

We end this section by recalling the following result proved in \cite{DFIZ1}, see Propositions 1.4 and 4.4 therein.

\begin{theorem}[\cite{DFIZ1}]\label{darre}
{ For each $\lambda>0$, let $u_\lambda$ be the viscosity solution of \eqref{bb}}. Let us assume that $H(x,0)\leq  c(H)$ for every $x\in M$. Then $u_\lambda\geq  0$ in $M$
for every $\lambda>0$ and $u_\lambda \nearrow u_0$ uniformly on $M$ as $\lambda\searrow 0$, where $u_0$ is the unique viscosity solution of \eqref{eq HJ check}
such that $u_0\equiv 0$ on $\A$.
\end{theorem}

As a corollary we infer

\begin{corollary}\label{cor Aubry}
For each $\lambda>0$, the function $u_\lambda$ is a viscosity subsolution of \eqref{eq HJ check}. In particular
\begin{itemize}
 \item[(i)] $u_\lambda\equiv 0$ on $\A$;
 \item[(ii)] $u_\lambda$ is differentiable on $\A$ and $d_x u_\lambda\equiv 0$ on $\A$.
\end{itemize}
\end{corollary}

\begin{proof}
From the fact that $u_\lambda\geq 0$ in $M$ we get
\[
H(x,-d_x u_\lambda) \leq \lambda u_\lambda+H(x,-d_x u_\lambda)=c(H)\qquad\hbox{in $M$}
\]
in the viscosity sense, namely $u_\lambda$ is a viscosity subsolution of \eqref{eq HJ check}. Item (i) follows from the inequality
\ $0\leq u_\lambda\leq u_0$ on $M$ and the fact that $u_0\equiv 0$ on $\A$, while item (ii) follows directly from Proposition \ref{prop Aubry}
and the fact that any constant function is a subsolution of \eqref{eq HJ check}.
\end{proof}
%
%
%\subsection{Notations}
%For reader's convenience, we list the notations used before and after here.
%\begin{itemize}
%	\item $\check{H}(x,p):=H(x,-p)$, $\check{L}(x,\dot{x})$ denotes the Fenchel transform of $\check{H}(x,p)$.
%	\item $\check{L}_\lambda(x,u,\dot{x}):=-\lambda u+\check{L}(x,\dot{x})$, $\check{H}_\lambda(x,u,p):=\lambda u+\check{H}(x,p)$.
%	\item $\pi:T^*M\times\R\to M$ and  $\pi^*:T^*M\times\R\to T^*M$ denote the standard projections.
%    \item $c(H)$ denotes the critical value of $H(x,p)$, there holds $c(H)=c(\check{H})$.
%	\item $\check{h}^{x_0,u_0}_\lambda(x,t)$ and $\check{T}^+_{t,\lambda}\varphi(x)$ denote the backward action function and the forward semigroup associated to $\check{L}_\lambda(x,u,\dot{x})+c(H)$.
%    \item $\check\Phi_\lambda^t$ denotes the contact Hamiltonian flow generated by $\check{H}_\lambda(x,u,p)-c(H)$.
%	\item ${u}_\lambda^+$ denotes the maximal forward weak KAM solution of \eqref{bb}, ${u}_\lambda$ denotes the unique backward weak KAM solution of \eqref{bb}.
%	\item $\tilde{\mathcal{M}}^\lambda$, $\tilde {\mathcal{A}}^\lambda$ denote the Mather set and Aubry set associated to $\check{H}_\lambda(x,u,p)-c(H)$ respectively.
%	\item ${\A}=\pi \tilde{\A}$ denotes the projected Aubry set associated to $H(x,p)$, there holds ${\A}=\check{\A}$.
%\end{itemize}
%

%
\section{Asymptotic convergence}

This section is devoted to the proof of Theorem \ref{222}. We prove some preliminary results first.

\begin{lemma}\label{uuee}
The family $\{u^+_\lambda\}_{\lambda\in (0,1]}$ is {equi-bounded} and equi-Lipschitz continuous.
\end{lemma}
\begin{proof}
First, we prove the equi-bounded character of  $\{u^+_\lambda\}_{\lambda\in (0,1]}$.
By Theorem \ref{teo forward sol}, $u^+_\lambda\leq u_\lambda$ on $M$ for each $\lambda\in [0,1]$.
By Theorem \ref{darre}, $u_\lambda $ converges uniformly to a function $u_0\in \Dr C(M,\R)$, in particular there exists $C>0$ such that $\|u_\lambda \|_\infty\leq C$ for each $\lambda\in [0,1]$. It follows that the  $\{u^+_\lambda\}_{\lambda\in (0,1]}$ are equi-bounded from above. Let us prove they are equi-bounded from below.
By Proposition \ref{wwyrr}, we have $u^+_\lambda=\check T_{t,\lambda}^+u^+_\lambda$ for each $t\geq 0$.
Take $x_0\in \mathcal{I}_{u^+_\lambda}=\{x\in M\ |\ u_{\lambda}(x)=u^+_{\lambda} (x)\}$.
Let $\alpha:[0,1]\to M$ be a geodesic connecting $x$ to $x_0$,  parameterized by constant speed $|\dot{\alpha}(s)|:=d(x,x_0)\leq \text{diam}(M)$ for each $s\in [0,1]$. Let
\[
C':=\max_{x\in M, |\dot{x}|\leq \text{diam}(M)} \big( \check{L}(x,\dot{x}) + c(H) \big).
\]
By (\ref{2-3}),
\begin{align*}
u^+_\lambda(x)=\check T_{1,\lambda}^+u^+_\lambda(x)\geq e^{\lambda}u_\lambda (x_0)-\int_0^1e^{\lambda s}\big( \check{L}(\alpha(s),\dot{\alpha}(s)) +c(H)\big)\,ds
\geq
e^\lambda \left( u_\lambda (x_0)-C'\right),
\end{align*}
which implies $u^+_\lambda$ is bounded from below for each $\lambda\in (0,1]$. Thus, there exists $K>0$ such that
$\|u^+_\lambda\|_\infty\leq K$ for all $\lambda\in (0,1]$.

Next, we prove the equi-Lipschitz continuity. For each $x$, $y\in M$, let $\beta:[0,d(x,y)]\to M$ be a geodesic of length $d(x,y)$, parameterized by arclength and connecting $x$ to $y$. Let
	{\[
	C'':=\sup\{\check{L}(x,\dot{x})+c(H)\ |\ x\in M,\ \|\dot{x}\|_x=1\}.
	\]
In view of \eqref{cali11}, we have
\[e^{\lambda d(x,y)}v(x)-e^{\lambda d(x,y)}v(y)\leq v(y)(1-e^{\lambda d(x,y)})+\frac{C''}{\lambda}(e^{\lambda d(x,y)}-1),\]
which yields from $0<\lambda\leq 1$ and $1-e^{-h}\leq h$ for all $h\in \R$,
\begin{align*}
v(x)-v(y)&\leq v(y)(e^{-\lambda d(x,y)}-1)+\frac{C''}{\lambda}(1-e^{-\lambda d(x,y)})\\
&\leq (C''+K)\frac{1-e^{-\lambda d(x,y)}}{\lambda}\\
&\leq (C''+K)d(x,y).
\end{align*}}
	We complete the proof  by exchanging the roles of $x$ and $y$.
\end{proof}

Next, we show the following result.

\begin{proposition}\label{prop Aubry in Sigma}
Let us assume that $H(x,0)\leq  c(H)$ for every $x\in M$. Then projected Aubry set $\A$ associated with \eqref{eq HJ check} is contained in ${\mathcal{A}}^\lambda$ for every $\lambda>0$.
\end{proposition}

\begin{remark}\label{oss Aubry in Sigma}
From the previous proposition and the fact that $d_x u_\lambda\equiv 0$ on $\A$ in view of Corollary \ref{cor Aubry}, we get in particular that
{\[
\{(y,0)\in T^*M\,\mid\,y\in\A\}\subseteq \tilde{\mathcal{A}}^\lambda.
\]}
\end{remark}

\begin{proof}
Let us fix $\lambda>0$ and pick $x\in\A$. According to Corollary \ref{cor Aubry}, the function $u_\lambda$ is differentiable at $x$, so
{$(x, d_x u_\lambda)\in G_{u_\lambda}$}, and by Proposition \ref{prop Aubry} there exists a unique curve
$\gamma:\R\to \A\subseteq M$ such that $\gamma(0)=x$ and
\[
u_\lambda(\gamma(b))-u_\lambda(\gamma(a))
=
\int_a^b  \big(\check L(\gamma(s),\dot\gamma(s))+c(H)\big)\,ds\qquad\hbox{for all $a<b$.}
\]
In particular,
\[
\check L(\gamma(s),\dot\gamma(s))+c(H)
=
\langle d_{\gamma(s)} u_\lambda,\dot\gamma(s)\rangle
=
\frac{d}{ds} u_\lambda(\gamma(s))\qquad\hbox{for all $s\in\R$.}
\]
By multiplying the above equality by $e^{\lambda s}$ and by integrating (by parts) in $(a,b)$ we get
\[
\int_a^b e^{\lambda s}  \big(\check L(\gamma(s),\dot\gamma(s))+c(H)\big)\,ds
=
e^{\lambda b} u_\lambda(\gamma(b)) - e^{\lambda a} u_\lambda(\gamma(a)) - \int_a^b \big(e^{\lambda s}\big)' u_\lambda(\gamma(s))\, ds,
\]
hence, by taking into account that $\gamma(\R)\subseteq \A$ and $u_\lambda\equiv 0$ on $\A$, we get
\[
e^{\lambda b} u_\lambda(\gamma(b)) - e^{\lambda a} u_\lambda(\gamma(a))
=
\int_a^b e^{\lambda s}  \big(\check L(\gamma(s),\dot\gamma(s))+c(H)\big)\,ds.
\]
This shows that $\gamma$ is the projection on $M$ of an integral curve of the {flow $\check\Phi^t_\lambda$ generated by (\ref{cott}), i.e.  $\gamma(t)=\pi\big(\check\Phi^t_\lambda(x,d_xu_\lambda)\big)$ } for all $t\in\R$.
In particular, we get that
$x\in \pi\big(\check\Phi_\lambda^{-t} (G_{u_\lambda})\big)$ for all $t\geq  0$, i.e. $x\in{\mathcal{A}}^\lambda$, as it was asserted.
\end{proof}

\begin{proof}[Proof of Theorem \ref{222}]
In view of Lemma \ref{uuee} and of {Ascoli-Arzel\'a's theorem}, it is enough to show that, if $u_{\lambda_n}^+$ converges to $u_*$ uniformly on $M$ as
$\lambda_n\ri 0^+$, then $u_*=u_0^+$ on $M$.  In view of the correspondence between forward and backward, or viscosity, solutions, see
\cite[Proposition 2.8]{WWY2}, we know that  $-u_{\lambda_n}^+$ is a viscosity solution of \eqref{aa}.
By the stability of the notion of viscosity solution, $-u_*$ is a viscosity solution of \eqref{ss}, which means that
$u_*$ is a forward weak KAM solution of (\ref{eq HJ check}). Furthermore,
by Corollary \ref{cor Aubry}, Theorem \ref{teo forward sol} and Proposition \ref{prop Aubry in Sigma}, we know that
$u^+_{\lambda_n}\equiv 0$ on $\A$, hence $u_*\equiv 0$ on $\A$.
Hence, $-u_*$ and $-u_0^+$ are both viscosity solutions of \eqref{ss} with  $-u_*\equiv -u_0^+$ on $\A$.
We conclude that $u_* \equiv u_0^+$ on $M$ by Proposition \ref{prop Aubry}.
\end{proof}

\section{On the example (\ref{ex-unique88}) }\label{e888}

By the recalled equivalence between viscosity, or backward weak KAM, solutions of \eqref{aa} and forward solutions of \eqref{bb}, see  \cite[Proposition 2.8]{WWY2},
it suffices to show the uniqueness of the forward weak KAM solution of
\begin{equation}\label{ex-unique}\tag{E$_I$}
\lambda u+\frac{1}{2}|d_x u|^2+U(x)=c,\quad x\in \T^1:=\mathbb{R}/\mathbb{Z},\quad \lambda>0,
\end{equation}
where $\mathbb{T}^1$ is a flat circle with the standard metric. For any two points $x,y\in M$, we use $|x-y|$ to denote the distance induced by the flat metric on $\mathbb{T}^1$.
We recall that $U:\T^1\to\R$ is of class $C^3$ and has a unique maximum point $x_0$ with $U(x_0)=c$, which is furthermore assumed to be non-degenerate,
i.e. $U''(x_0)<0$. When $U(x)=\cos(2\pi x)$,
(\ref{ex-unique}) corresponds to the dissipative pendulum.

To obtain this uniqueness result, we need some preliminary material that we will develop in the next section.

\subsection{Some preliminary facts}

We start by the following known fact about reversible Hamiltonians.

\begin{proposition}\label{prop reversible}
Let us assume that $H$ is reversible, i.e. $H(x,p)=H(x,-p)$ for all $(x,p)\in T^*M$.
Then
\begin{itemize}
\item[(i)] $H(x,p)>H(x,0)$ for every $x\in M$ and $|p|\not=0$, moreover $c(H)=\max_{x\in M} H(x,0)$, in particular, \ $H(x,0)\leq  c(H)$\  for every $x\in M$;
\item[(ii)] the projected Aubry set $\A$ associated with \eqref{ss} is given by
$$
\A=\{y\in M\,\mid\, H(y,0)=c(H)\};
$$
\end{itemize}
\end{proposition}

\begin{proof}
Let us prove (i). Since $H$ is reversible, we have
$\frac{\partial H}{\partial p}(x,p)=-\frac{\partial H}{\partial p}(x,-p)$\ \ for each $(x,p)\in T^* M$.
In particular, $\frac{\partial H}{\partial p}(x,0)=0$. Combining with $\frac{\partial^2 H}{\partial p^2}(x,p)>0$, we have
for each $x\in M$,
\begin{equation}\label{eq reversible H}
H(x,p)>H(x,0)=\min_{p\in T_x^*M}H(x,p)\qquad\hbox{for all $x\in M$ and $p\in T_x^*M\setminus\{0\}$}.
\end{equation}
Let us set $c:=\max_{x\in M} H(x,0)$. Since any constant function $v$ on $M$ is a subsolution of $H(x,d_x v)=c$ in $M$, we have $c(H)\leq c$ in view of
\eqref{def critical constant}. On the other hand, if $v$ is subsolution of $H(x,d_x v)=c(H)$ in $M$, we have in particular
\[
H(x,0)=\min_{p\in T_x^* M} H(x,p)\leq H(x,d_x v)\leq c(H)\qquad\hbox{for a.e. $x\in M$,}
\]
yielding $c=\max_{x\in M} H(x,0)\leq c(H)$. This shows that $\max_{x\in M} H(x,0)=c(H)$.

(ii) Let us denote by $\mathcal E$ the set appearing at the right-hand side of the equality in (ii).
The function $v_0\equiv 0$ satisfies $H(x,d_x v_0)\leq c(H)$ for every $x\in M$, with a strict inequality holding when $x\not\in\mathcal E$.
This shows that $\A\subseteq \mathcal E$ in view of Proposition \ref{prop Aubry}. Let now $v$ be a subsolution of $H(x,d_x v)=c(H)$ in $M$. Then
the Clarke generalized gradient $\partial^c v(x)$ of $v$ at $x$ satisfies \ \ $\partial^c v(x)\subseteq \{p\in T_x^*M\,\mid\, H(x,p)\leq c(H)\}$
for every $x\in M$, see for instance \cite{Sic_chapter}. In particular, it is a singleton whenever $x\in\mathcal E$. This implies that $v$ is (strictly)
differentiable at any $x\in\mathcal E$, see \cite[Proposition 2.2.4]{Cl}. This show that $\mathcal E\subseteq \A$ by \eqref{def Aubry}.
\end{proof}

We focus now on the properties enjoyed by the Mather set $\tilde \M^\lambda$ associated with equation \eqref{ex-unique}. We start with a general fact for reversible Hamiltonians.

\begin{proposition}\label{prop invariant measure}
Let us assume that $H$ is reversible and let $\mu$ be a $\check\Phi_\lambda^t$-invariant probability measure on $T^*M$. Then
\[
 \text{supp}(\mu)\subseteq\{(x,0)\in T^*M\,\mid\,x\in M\}.
\]
\end{proposition}
\begin{proof}
Fix $(x,p)\in T^*M$. For every $t\in\R$ we have
\[
 \frac{d}{dt} H\left(\check\Phi^t_\lambda(x,p)\right)
 =
 -\lambda \langle p(t), \frac{\partial \check H}{\partial p}\left(\check\Phi^t_\lambda(x,p)\right)\rangle \leq 0
\]
by convexity, with equality holding if and only if $p(t)=0$ in view of \eqref{eq reversible H}. If $\mu$ is $\check\Phi_\lambda^t$-invariant,
we infer
\begin{eqnarray*}
0
=
\int_{T^*M} H\left(\check\Phi^1_\lambda(x,p)\right)\, d\mu(x,p)
-
\int_{T^*M} H\left(\check\Phi^t_\lambda(x,p)\right)\, d\mu(x,p)\\
= \int_0^1 \left( \int_{T^*M} \frac{d}{ds} H\left(\check\Phi^s_\lambda(x,p)\right)\, d\mu(x,p)\right)\, ds
=
-\lambda  \int_{T^*M} \langle p, \frac{\partial \check H}{\partial p}\left(x,p\right) \, d\mu(x,p),
\end{eqnarray*}
yielding
\[
 \langle p, \frac{\partial \check H}{\partial p}\left(x,p\right)\rangle=0\qquad\hbox{for $\mu$--a.e. $(x,p)\in T^*M$.}
\]
The assertion follows in view of \eqref{eq reversible H}.
\end{proof}

We exploit the previous result to derive the following information.

\begin{proposition}\label{prop disreversible}
Let $\tilde\M^\lambda$ be the Mather set associated with equation \eqref{ex-unique}, where
$U:\T^1\to\R$ is of class $C^3$ and has a unique maximum point $x_0$ with $U(x_0)=c$, which is furthermore assumed to be non-degenerate,
i.e. $U''(x_0)<0$. Then $(x_0,0)$ is a hyperbolic fixed point for the discounted flow \eqref{cott} and, for $\lambda>0$ small enough,
$\tilde\M^\lambda=\{(x_0,0)\}$.
\end{proposition}
\begin{proof}
The fact that $(x_0,0)$ a hyperbolic fixed point for the discounted flow \eqref{cott} is easily checked. For every fixed $\lambda>0$, let us pick
$(x_\lambda,0)\in \tilde{\mathcal{M}}^\lambda$ and set $(x_\lambda(t),p_\lambda(t)):=\check\Phi^s_\lambda(x_\lambda,0)$.
By taking into account Proposition \ref{prop invariant measure} together with $\tilde{\mathcal{M}}^\lambda\subseteq \tilde{\mathcal{I}}_{u_\lambda}$ and
Proposition \ref{wwyrr}, we infer that
\[
0=p_\lambda(t)=d_{x_\lambda(t)}u_\lambda,
\quad \dot x_\lambda(t)=p_\lambda(t)=0,
\quad
0=\dot p_\lambda(t)=-\frac{\partial U}{\partial x}(x_\lambda(t),p_\lambda(t))
\]
for all $t\in\R$, namely $x_\lambda(t)=x_\lambda$ for all $t\in\R$ and $x_\lambda$ is a critical point for $U$. Furthermore,
the static curve $x_\lambda(t)=x_\lambda$ for all $t\in\R$
is optimal for $u_\lambda(x_\lambda)$, i.e. for any $a\leq b$,
 \[e^{\lambda b}u_\lambda(x_\lambda(b))-e^{\lambda a} u_\lambda(x_\lambda(a))
 = \int_a^b e^{\lambda s}\big( \check L(x_\lambda(s),\dot{x_\lambda}(s))+c \big) ds.
 \]
By taking $b=0$ and $a\ri-\infty$ we get
\[
u_{\lambda}(x_\lambda)=\int_{-\infty}^0 e^{\lambda s}\big(\check L(x_\lambda,0)+c\big)\,ds=\frac{c-U(x_\lambda)}{\lambda}\geq 0.
\]
Since $u_\lambda$ is converging to $u_0$ as $\lambda\to 0^+$ by Theorem \ref{darre},
we necessarily have that $x_\lambda \to x_0$ as $\lambda\to 0$. Since $x_0$ is an isolated critical point by the non-degeneracy condition, we infer that there exists $\lambda_0>0$
such that $x_\lambda=x_0$ for every $\lambda \leq \lambda_0$.
\end{proof}

\begin{remark}
In the dissipative pendulum case, i.e. when $U(x)=\cos(2\pi x)$, the statement of Proposition \ref{prop disreversible} holds for any $\lambda>0$. 
This follows as a consequence of  
\cite[Example 2]{MS} together with Theorem \ref{teo forward sol}-(iii). This is not true for a potential $U$ of more general form. In fact, let $w(x)$ be a smooth function on 
$\T^{1}$. Let $x_0$ be  its unique global minimum point with $w(x_0)=0$, $w''(x_0)>0$ and let  $x_1$ be another critical point such that $w'(x_1)=0$ and $w(x_1)>0$. Set
$U(x):=-w(x)-\frac{1}{2}|d_xw|^2$.
Then $x_0$ is the unique global maximum point of $U(x)$ with $U(x_0)=0$, $U''(x_0)<0$. For $\lambda_0=1$, $w$ is a smooth solution  
(hence both forward and backward weak KAM solution) of 
\[\lambda_0 u+\frac{1}{2}|d_x u|^2+U(x)=0.\] 
However,  one has $\{(x_0,0),(x_1,0)\}\subseteq\tilde{\mathcal{M}}^{\lambda_0}$.
\end{remark}

%The following propositions were shown in
%\cite{WWY3,WWY3i} in a more general setting, we add  proofs here for the sake of completeness.

\subsection{Uniqueness of the forward weak KAM solution}
Let us now come back to the analysis of equation \eqref{ex-unique}. In view of Proposition \ref{prop disreversible}, the uniqueness of the forward weak KAM solution
to equation \eqref{ex-unique} is a consequence of the following more general result.

\begin{proposition}
Let $H:T^*\T^1\to\R$ be a $C^3$--Hamiltonian, satisfying hypotheses (H1)-(H2). Let us assume that the Mather set $\tilde \M^\lambda$ associated with
the discounted Hamilton-Jacobi equation
\begin{equation}\label{last HJ eq}
 \lambda u+H(x,-d_x u)=c(H)\qquad\hbox{in $\T^1$}
\end{equation}
reduces to $\{(x_0,0)\}$ and that $(x_0,0)$ is a hyperbolic fixed point for the discounted flow generated by \eqref{cott}. Then equation \eqref{last HJ eq} admits a unique
forward weak KAM solution.
\end{proposition}

\begin{proof}
Let us denote by $\mathcal{S}_\lambda^+$ be the set of all forward weak KAM solutions of (\ref{last HJ eq}).
For each $v\in \mathcal{S}_\lambda^+ $, we have $\tilde{\mathcal{M}}^\lambda_{v }\subseteq \tilde{\mathcal{M}}^\lambda$. Since $\tilde{\mathcal{M}}^\lambda_{v }$ is nonempty, we
necessarily have
$$
\tilde{\mathcal{M}}^\lambda_{v }=\{(x_0,0)\}.
$$
This means $v (x_0)=0$ and $d_{x_0} v=0$.

We denote for simplicity, the stable submanifold of $(x_0,0)$ with respect to $\check\Phi_\lambda^t$ by
\[W^s(x_0,0):=\{(x,p)\in T^*\T^1\ |\ \lim_{t\rightarrow+\infty}d(\check\Phi_\lambda^t(x,p),(x_0,0))=0\},\]
where $d(\cdot,\cdot)$ is a Riemannian metric on $T^*\T^1$.
Given $\eps>0$, we denote the local stable submanifold of $(x_0,0)$ with respect to $\check\Phi_\lambda^t$ by
\[W_\eps^s(x_0,0):=\{(x,p)\in W^s(x_0,0)\ |\ d(\check\Phi_\lambda^t(x,p),(x_0,0))<\eps,\ \forall t\geq 0\}.\]
Since $\check{H}\in C^3$, by the stable manifold theorem \cite{HPS}, there exist $\delta>0$, $h\in C^2$ with $h(x_0)=0$ such that
\[[x_0-\delta,x_0+\delta]\subseteq \pi \big( W_\eps^s(x_0,0),\quad W_\eps^s(x_0,0)=\{(x,h(x))\ |\ x\in [x_0-\delta,x_0+\delta]\},\]
where $\pi :T^*\T^1\rightarrow \T^1$ denotes the standard projection.

For each $v\in \mathcal{S}_\lambda^+ $, $v$ is Lipschitz continuous. Furthermore, it is semiconvex. Indeed, $-v$ is a solution to
\[
H(x,d_xu)=c-\lambda v(x)\qquad \hbox{in $\T^1$,}
\]
so it is semiconcave in view of the results in \cite{CS89}.
Denote $\mathcal{D}$ be the set of all differentiable points of $v$.
Pick $\bar{x}\in \mathcal{D}\cap (x_0-1,x_0)$ and set $\bar{p}:=d_{\bar{x}}v$. Let $(x(t),p(t)):=\check\Phi_\lambda^t(\bar{x},\bar{p})$ for all $t\geq 0$.
By Proposition \ref{wwyrr}, $d_{x(t)}v=p(t)$ for all $t\geq 0$. In view of Proposition \ref{asypp}, $(x_0,0)\in \omega(\bar{x},\bar{p})$. Thus,
there exists a diverging sequence $(t_n)_n$ such that either $x(t_n)\rightarrow x_0$ or $x(t_n)\to x_0-1$ as $t_n\rightarrow+\infty$.
We assert that the map $t\mapsto x(t)$ is monotone in $[0,+\infty)$. Let us assume for definiteness that $x(t_n)\rightarrow x_0$.
We claim that $\dot x(t)\geq 0$ for all $t>0$. In fact, let us assume by contradiction that $\dot{x}(t_0)<0$ for some $t_0>0$.
By the fact that $x(t_n)\rightarrow x_0$, one can find $t_1>t_0$ such that $x(t_0)=x(t_1)=:\hat{x}$ and $\dot{x}(t_1)\geq 0$, a contradiction to the fact that
\[
\dot x (t_0)=\frac{\partial\check H}{\partial p}(\hat x,d_{\hat x} v)=\dot x(t_1).
\]
The monotonicity of $t\mapsto x(\cdot)$ implies that $x(t)\rightarrow x_0$. Furthermore, $v$ is differentiable at $x_0$ and at $x(t)$ for every $t>0$. By semiconvexity,
we infer that $d_{x(t)}v\rightarrow d_{x_0}v=0$ as $t\to +\infty$, i.e.
\[
\check\Phi_\lambda^t(\bar{x},\bar{p})\rightarrow (x_0,0)\qquad \text{as}\ t\rightarrow+\infty.
\]
That implies $(\bar{x},\bar{p})\in W^s(x_0,0)$.  For each  $\tilde{x}\in \mathcal{D}\cap [x_0-\delta,x_0+\delta]$, we have $(\tilde{x},d_{\tilde{x}}v)\in W_\eps^s(x_0,0)$. Moreover, $d_xv=h$ on $\mathcal{D}\cap [x_0-\delta,x_0+\delta]$. Note that $\mathcal{D}$ has full Lebesgue measure on $\T^1$. It follows that, for each $x\in[x_0-\delta,x_0+\delta]$,
\[v(x)=\int_{x_0}^xd_yv(y) dy=\int_{x_0}^x h(y) dy.\]
This shows $d_xv=h$ on $[x_0-\delta,x_0+\delta]$.

For any $v_1,v_2\in \mathcal{S}_\lambda^+$, we have $v_1(x_0)=v_2(x_0)=0$ and $d_xv_1=d_xv_2=h$ on $[x_0-\delta,x_0+\delta]$.
That yields $v_1=v_2$ on $[x_0-\delta,x_0+\delta]$. By Proposition \ref{copp}, $v_1=v_2$ on $\T^1$, namely \eqref{last HJ eq} has a unique forward weak KAM solution.
\end{proof}

%====================================================================================================================

\vskip 1cm

\noindent {\bf Acknowledgements:}
The authors would like to thank Professor Antonio Siconolfi and Professor  Jun Yan  for very helpful conversations on this topic.  They also would like to
thank an anonymous referee for pointing out an error in the proof and in the statement of the submitted version of Proposition \ref{prop disreversible}  
and for several comments and suggestions that helped
to shorten up and improve the presentation.  Lin Wang was supported by NSFC Grant No. 11790273, 11631006.

\medskip

\bibliography{discounted}

\begin{thebibliography}{10}

\bibitem{barles}
{\sc G.~Barles}, {\em Solutions de viscosit\'e des \'equations de
  {H}amilton-{J}acobi}, vol.~17 of Math\'ematiques \& Applications (Berlin)
  [Mathematics \& Applications], Springer-Verlag, Paris, 1994.

\bibitem{BernardC11}
{\sc P.~Bernard}, {\em Existence of {$C^{1,1}$} critical sub-solutions of the
  {H}amilton-{J}acobi equation on compact manifolds}, Ann. Sci. \'Ecole Norm.
  Sup. (4), 40 (2007), pp.~445--452.

\bibitem{BernardSmooth}
\leavevmode\vrule height 2pt depth -1.6pt width 23pt, {\em Smooth critical
  sub-solutions of the {H}amilton-{J}acobi equation}, Math. Res. Lett., 14
  (2007), pp.~503--511.

\bibitem{CCJWY}
{\sc P.~{Cannarsa}, W.~{Cheng}, L.~{Jin}, K.~{Wang}, and J.~{Yan}}, {\em
  {Herglotz' variational principle and Lax-Oleinik evolution}}, J. Math. Pures
  Appl.,  (to appear).

\bibitem{CCY}
{\sc P.~Cannarsa, W.~Cheng, K.~Wang, and J.~Yan}, {\em Herglotz' generalized
  variational principle and contact type {H}amilton-{J}acobi equations}, in
  Trends in control theory and partial differential equations, vol.~32 of
  Springer INdAM Ser., Springer, Cham, 2019, pp.~39--67.

\bibitem{CS89}
{\sc P.~Cannarsa and H.~M. Soner}, {\em Generalized one-sided estimates for
  solutions of {H}amilton-{J}acobi equations and applications}, Nonlinear
  Anal., 13 (1989), pp.~305--323.

\bibitem{CCIZ}
{\sc Q.~Chen, W.~Cheng, H.~Ishii, and K.~Zhao}, {\em Vanishing contact
  structure problem and convergence of the viscosity solutions}, Comm. Partial
  Differential Equations, 44 (2019), pp.~801--836.

\bibitem{Cl}
{\sc F.~H. Clarke}, {\em Optimization and nonsmooth analysis}, Canadian
  Mathematical Society Series of Monographs and Advanced Texts, John Wiley \&
  Sons Inc., New York, 1983.
\newblock A Wiley-Interscience Publication.

\bibitem{CIPP}
{\sc G.~Contreras, R.~Iturriaga, G.~P. Paternain, and M.~Paternain}, {\em
  Lagrangian graphs, minimizing measures and {M}a\~{n}\'{e}'s critical values},
  Geom. Funct. Anal., 8 (1998), pp.~788--809.

\bibitem{DFIZ1}
{\sc A.~Davini, A.~Fathi, R.~Iturriaga, and M.~Zavidovique}, {\em Convergence
  of the solutions of the discounted {H}amilton-{J}acobi equation: convergence
  of the discounted solutions}, Invent. Math., 206 (2016), pp.~29--55.

\bibitem{Fa12}
{\sc A.~Fathi}, {\em Weak {KAM} from a {PDE} point of view: viscosity solutions
  of the {H}amilton-{J}acobi equation and {A}ubry set}, Proc. Roy. Soc.
  Edinburgh Sect. A, 142 (2012), pp.~1193--1236.

\bibitem{Fat-b}
\leavevmode\vrule height 2pt depth -1.6pt width 23pt, {\em Weak {KAM} {T}heorem
  in {L}agrangian {D}ynamics, preliminary version 10, {L}yon}.
\newblock unpublished, June 15 2008.

\bibitem{FSC1}
{\sc A.~Fathi and A.~Siconolfi}, {\em Existence of {$C^1$} critical
  subsolutions of the {H}amilton-{J}acobi equation}, Invent. Math., 155 (2004),
  pp.~363--388.

\bibitem{FS05}
\leavevmode\vrule height 2pt depth -1.6pt width 23pt, {\em P{DE} aspects of
  {A}ubry-{M}ather theory for quasiconvex {H}amiltonians}, Calc. Var. Partial
  Differential Equations, 22 (2005), pp.~185--228.

\bibitem{Go}
{\sc D.~A. Gomes}, {\em Generalized {M}ather problem and selection principles
  for viscosity solutions and {M}ather measures}, Adv. Calc. Var., 1 (2008),
  pp.~291--307.

\bibitem{Go1}
{\sc D.~A. {Gomes}, H.~{Mitake}, and H.~V. {Tran}}, {\em {The Selection problem
  for discounted Hamilton-Jacobi equations: some non-convex cases}}, ArXiv
  e-prints,  (2016).

\bibitem{HPS}
{\sc M.~W. Hirsch, C.~C. Pugh, and M.~Shub}, {\em Invariant manifolds}, Lecture
  Notes in Mathematics, Vol. 583., Springer-Verlag, Berlin-New York, 1977.

\bibitem{IsMiTr-discount1}
{\sc H.~Ishii, H.~Mitake, and H.~V. Tran}, {\em The vanishing discount problem
  and viscosity {M}ather measures. {P}art 1: {T}he problem on a torus}, J.
  Math. Pures Appl. (9), 108 (2017), pp.~125--149.

\bibitem{IsMiTr-discount2}
\leavevmode\vrule height 2pt depth -1.6pt width 23pt, {\em The vanishing
  discount problem and viscosity {M}ather measures. {P}art 2: {B}oundary value
  problems}, J. Math. Pures Appl. (9), 108 (2017), pp.~261--305.

\bibitem{KB}
{\sc N.~Kryloff and N.~Bogoliuboff}, {\em La th\'eorie g\'en\'erale de la
  mesure et son application \`a l'\'etude des syst\`emes dynamiques de la
  m\'ecanique non lin\'eaire}, Ann. Math. II. S\'er., 38 (1937), pp.~65--113.

\bibitem{MS}
{\sc S.~Mar\`o and A.~Sorrentino}, {\em Aubry-{M}ather theory for conformally
  symplectic systems}, Comm. Math. Phys., 354 (2017), pp.~775--808.

\bibitem{MK}
{\sc H.~Mitake and K.~Soga}, {\em Weak {KAM} theory for discounted
  {H}amilton-{J}acobi equations and its application}, Calc. Var. Partial
  Differential Equations, 57 (2018), pp.~Art. 78, 32.

\bibitem{MiTr-discount}
{\sc H.~Mitake and H.~V. Tran}, {\em Selection problems for a discount
  degenerate viscous {H}amilton--{J}acobi equation}, Adv. Math., 306 (2017),
  pp.~684--703.

\bibitem{Sic_chapter}
{\sc A.~Siconolfi}, {\em Hamilton-{J}acobi equations and weak {KAM} theory}, in
  Mathematics of complexity and dynamical systems. {V}ols. 1--3, Springer, New
  York, 2012, pp.~683--703.

\bibitem{SWY}
{\sc X.~Su, L.~Wang, and J.~Yan}, {\em Weak {KAM} theory for
  {H}amilton-{J}acobi equations depending on unknown functions}, Discrete
  Contin. Dyn. Syst., 36 (2016), pp.~6487--6522.

\bibitem{WWY}
{\sc K.~Wang, L.~Wang, and J.~Yan}, {\em Implicit variational principle for
  contact {H}amiltonian systems}, Nonlinearity, 30 (2017), pp.~492--515.

\bibitem{WWY2}
\leavevmode\vrule height 2pt depth -1.6pt width 23pt, {\em Aubry-{M}ather
  theory for contact {H}amiltonian systems}, Comm. Math. Phys., 366 (2019),
  pp.~981--1023.

\bibitem{WWY1}
\leavevmode\vrule height 2pt depth -1.6pt width 23pt, {\em Variational
  principle for contact {H}amiltonian systems and its applications}, J. Math.
  Pures Appl. (9), 123 (2019), pp.~167--200.

\bibitem{WY}
{\sc Y.-N. Wang and J.~Yan}, {\em A variational principle for contact
  {H}amiltonian systems}, J. Differential Equations, 267 (2019),
  pp.~4047--4088.

\bibitem{ZC}
{\sc K.~Zhao and W.~Cheng}, {\em On the vanishing contact structure for
  viscosity solutions of contact type {H}amilton-{J}acobi equations {I}:
  {C}auchy problem}, Discrete Contin. Dyn. Syst., 39 (2019), pp.~4345--4358.

\end{thebibliography}
\bibliographystyle{siam}

\end{document}